\documentclass{amsart}

\usepackage{amsmath}
\usepackage{amssymb}
\usepackage{amsthm}

\usepackage{enumerate}

\usepackage{graphicx}
\usepackage{float}

\theoremstyle{plain}
\newtheorem{theorem}{Theorem}

\newtheorem{prop}[theorem]{Proposition}
\newtheorem{lemma}[theorem]{Lemma}

\newtheorem{fact}[theorem]{Fact}

\theoremstyle{definition}

\newtheorem{rem}[theorem]{Remark}
\newtheorem{exmp}[theorem]{Example}

\DeclareMathOperator{\sts}{STS}
\DeclareMathOperator{\psts}{PSTS}

\begin{document}

\title{Point-line geometries related to binary equidistant codes}
\author{Mark Pankov, Krzysztof Petelczyc, Mariusz \.Zynel}
\keywords{simplex code, equidistant code, Pasch configuration, Cremona-Richmond configuration, partial Steiner triple system}
\subjclass[2020]{51E20,51E22}
\address{Mark Pankov: Faculty of Mathematics and Computer Science, 
University of Warmia and Mazury, S{\l}oneczna 54, 10-710 Olsztyn, Poland}
\email{pankov@matman.uwm.edu.pl}
\address{Krzysztof Petelczyc, Mariusz \.Zynel: Faculty of Mathematics, University of Bia{\l}ystok, Cio{\l}kowskiego 1M, 15-245 Bia{\l}ystok, Poland}
\email{kryzpet@math.uwb.edu.pl, mariusz@math.uwb.edu.pl}
\maketitle

\begin{abstract}
Point-line geometries whose singular subspaces correspond to binary equidistant codes
are investigated. 
The main result is a description of automorphisms of these geometries. 
In some important cases, automorphisms induced by non-monomial linear automorphisms surprisingly arise.
\end{abstract}                                             

\section{Introduction}
Graphs and geometries related to various types of linear codes are considered in \cite{CGK,CG,KP1,KP2,KPP,KP-dim2,Pank}.
In the present paper, we investigate point-line geometries whose maximal singular subspaces correspond to equivalence classes of binary equidistant 
codes (not necessarily non-degenerate). 

The main examples of equidistant codes are simplex codes.
These codes are interesting for many reasons.
They are dual to Hamming codes and are first-order Reed-Muller codes \cite[Subsection 1.2.2]{TVN}. 
The name comes from the property that the code words of a binary $k$-dimensional simplex code form a $(2^k-1)$-simplex of constant edge length if the code words are interpreted as points of ${\mathbb R}^{2^k-1}$
(simplex codes over fields with more than two elements lose this property). 
Generally, a $q$-ary simplex code of dimension $k$ is a non-degenerate linear code of length $\frac{q^k-1}{q-1}$ and the Hamming distance between any two distinct code words is $q^{k-1}$; all such codes are equivalent by the MacWilliams theorem. 
Also, all code words of these codes form an algebraic surface
which is quadric only in two cases: $q=2$, $k=3$ and $q=3$,
$k=2$ \cite{KPP}.
Every non-degenerate equidistant code is 
equivalent to some replication of a simplex code \cite{Bonis}. A degenerate equidistant code can be obtained from a non-degenerate equidistant code by adding some zero coordinates to each code word.
So, equivalence classes of equidistant codes are completely determined by the length, dimension and degeneracy degree (under the assumption that the ground field is fixed). 

%The Hamming weights of code words of binary equidistant codes are even.
A geometry defined by an equivalence class of binary equidistant codes
is formed by all points of a projective space 
(over the two-element field)   
with a fixed even Hamming weight. 
Two such points are collinear if the Hamming distance between them is equal to their Hamming weight.
So, points of the geometry can be naturally identified with $2m$-element subsets of an $n$-element set and two distinct points are collinear if  the intersection of the corresponding subsets contains precisely $m$ elements; in this case, the third point on the line is determined by the symmetric difference of these subsets. 
Then $n\ge 3m$ since otherwise, there are no lines. %(otherwise, lines are not defined).
Here are some known geometries from the infinite family described above:
\begin{itemize}
\item 
a line of size $3$ ($n=3m=3$),

\item 
the Pasch configuration ($n=3m+1=4$),

\item 
the Cremona-Richmond configuration known also as the generalized quadrangle of type $2,2$ ($n=3m=6$),

\item
if $n=4m-1=2^{k}-1$, then the maximal singular subspaces
of the geometry correspond to binary simplex codes of dimension $k$; for $m=2$ we obtain a polar space.
\end{itemize}
The main result (Theorem \ref{theorem-aut}) states that  automorphisms of the geometry are induced by coordinate permutations (monomial linear automorphisms) of the corresponding vector space except the cases when $n=4m-1$ or $n=4m$ for which additional automorphisms induced by non-monomial linear automorphisms arise. 
Note that the first exceptional case includes the geometry of simplex codes.

If $n\ge 3m+1$ and $m\ne2$, then
the collinearity graph of our geometry
contains maximal cliques which are not maximal singular %different from maximal singular
subspaces. It follows from \cite{Ryser} that 
%such a clique 
%contains no more than $n$ elements and these elements are blocks of
%a certain symmetric block design if the number of elements is equal to $n$.
the cardinality of such a clique is not greater than $n$
and in case it is $n$ its elements are blocks of a certain symmetric block design.

The geometries in question can be also considered as 
partial Steiner triple systems ($\psts$) embedded in 
Steiner triple systems ($\sts$) which are projective spaces over the two-element field. 
There are many papers devoted to embeddings of $\psts$'s in $\sts$'s, e.g.\ \cite{embPSTS}.
Our result is related to \cite{Cameron}, where embeddings of $\psts$'s in $\sts$'s satisfy the property that every $\psts$ automorphism can be extended to an $\sts$ automorphism (Remark \ref{rem-psts}). 
%Pasch configuration is the relevant example of $\psts$ %which is used to investigate  $\sts$'s. It plays an %important role also in in our reasonings. Roughly %speaking, we consider the Fano plane as an $\sts$ obtained %by the complement of two Pasch configurations differing in %one point (Figure~\ref{fig:veblenplus}).

%It will be also interesting to consider point-line gometries related to non-binary equidistant codes. 
The geometry of $4$-ary simplex codes of dimension $2$ is investigated in \cite{KP-dim2}.
The non-binary case is completely different and does not admit a set-theoretic interpretation. A point is not determined by its set of non-zero coordinates; furthermore, two points of the same Hamming weight need not to be collinear in the corresponding geometry if the Hamming distance between them is equal to their Hamming weight (Section 5).

\section{Main objects, examples and basic properties}

A {\it point-line geometry} is a pair $({\mathcal P},{\mathcal L})$,
where ${\mathcal P}$ is a set whose elements are called {\it points} and 
${\mathcal L}$ is a family of subsets of ${\mathcal P}$ called {\it lines}. 
Every line contains at least two points and the intersection of two distinct lines contains at most one point.
Two distinct points are said to be {\it collinear} if there is a line containing them.
The {\it collinearity graph} of $({\mathcal P},{\mathcal L})$
is the simple graph whose vertex set is ${\mathcal P}$ and two distinct vertices are connected by an edge if they are collinear points. 
A subset $X\subset {\mathcal P}$ is called a {\it subspace} if for any two distinct collinear points from $X$
the line containing these points is a subset of $X$. 
A subspace is {\it singular} if any two distinct points of this subspace are collinear. 

Let ${\mathbb F}$ be the two-element field. 
Consider the $n$-dimensional vector space $V={\mathbb F}^n$ over this field with 
the standard basis 
$$e_{1}=(1,0,\dots,0),\dots,e_{n}=(0,\dots,0,1).$$
Every non-zero vector of $V$ is of type $e_I=\sum_{i\in I}e_i$,
where $I$ is a non-empty subset of $[n]=\{1,\dots,n\}$;
the $i$-th coordinate of $e_I$ is $1$ if $i\in I$ and $0$ otherwise.
Let ${\mathcal P}(V)$ be the associated projective space and let  $P_I$
be the point of ${\mathcal P}(V)$ corresponding to $e_I$.
Recall that $|I|$ is called the {\it Hamming weight} of $P_I$.
The line of ${\mathcal P}(V)$ containing two distinct points $P,Q\in {\mathcal P}(V)$
is denoted by $\langle P, Q\rangle$ and  we write $P\odot Q$ for the point of this line distinct from $P,Q$.
For any non-empty subsets $I,J\subset [n]$ we have 
$$e_I+e_J=e_{I\mathbin{\triangle} J}$$
and, consequently, 
$$P_I\odot P_J= P_{I\mathbin{\triangle} J}.$$
Let $m$ be a positive integer satisfying $3m\le n$. 
Denote by ${\mathcal P}_m$ the set of all points $P_I\in {\mathcal P}(V)$ such that $|I|=2m$.
For distinct $P_I,P_J\in {\mathcal P}_m$ the point $P_I\odot P_J$ belongs to ${\mathcal P}_m$ if and only 
if $|I\cap J|=m$. 
We will consider ${\mathcal P}_m$ as a point-line geometry whose lines are the lines of ${\mathcal P}(V)$
contained in ${\mathcal P}_m$.

\begin{rem}
If $3m>n$, then for any subsets $I,J\subset [n]$ satisfying 
$|I|=|J|=2m$ we have $|I\cap J|>m$ which implies that $|I\mathbin{\triangle} J|< 2m$, i.e.\
the set of all $P_I$, $|I|=2m$ contains no line of ${\mathcal P}(V)$.
The same holds for every set formed by all points of fixed odd Hamming weight 
(if $I,J\subset [n]$ and $|I|=|J|$ is odd, then $|I\mathbin{\triangle} J|$  is even). 
\end{rem}

Note that if $n=3$, then ${\mathcal P}_1$ is a line of size $3$.
If $n=4$, then ${\mathcal P}_1$ is the Pasch (Veblen) configuration, see Fig.~\ref{fig:veblen}.
\begin{figure}[ht]
\begin{center}
\includegraphics[scale=0.65]{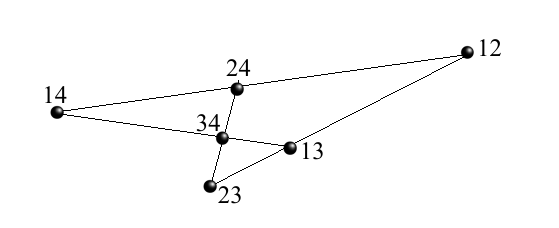}
\caption{The Pasch configuration: points and lines of ${\mathcal P}_1$ for $n=4$.}
\label{fig:veblen}
\end{center}
\end{figure}  

If $n=6$, then ${\mathcal P}_2$ is the Cremona-Richmond configuration, see Fig.~\ref{fig:crconf}. Every $P_I\in {\mathcal P}_2$ is identified with the $2$-element subset $[6]\setminus I$ and three points of ${\mathcal P}_2$ form a line if and only if the corresponding $2$-element subsets are mutually disjoint.

\begin{figure}[ht]
\begin{center}
\includegraphics[scale=0.45]{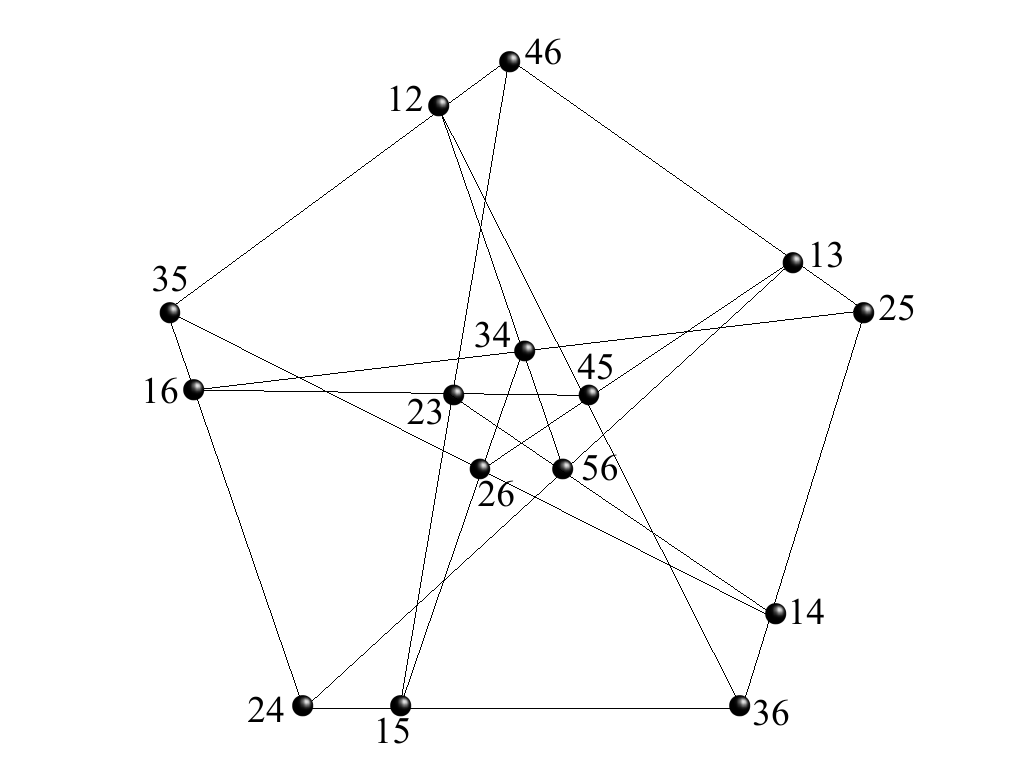}
\caption{The Cremona-Richmond configuration: points and lines of ${\mathcal P}_2$ for $n=6$.}
\label{fig:crconf}
\end{center}
\end{figure}  

Let $\Gamma_m$ be the collinearity graph of the geometry ${\mathcal P}_m$.
Observe that $\Gamma_1$ is the Johnson graph $J(n,2)$.
The following statement shows that $\Gamma_m$ is a connected graph of diameter $2$.

\begin{prop}\label{prop-coll-graph}
For any two points of\/ ${\mathcal P}_m$
there is a point of\/ ${\mathcal P}_m$ collinear to each of them.
\end{prop}

\begin{proof}
The statement is obvious for collinear points of ${\mathcal P}_m$.
Let $P_I,P_J$ be non-collinear points of ${\mathcal P}_m$. Then $|I\cap J|\ne m$.

If $|I\cap J|<m$, then 
$$|I\setminus J|=|J\setminus I|>m.$$
In this case, we take any $m$-element subsets 
$$M\subset I\setminus J\quad\text{and}\quad N\subset J\setminus I.$$
The point $P_{M\cup N}$ is as required.

Now, suppose that $|I\cap J|>m$. 
Then 
$$|I\cap J|=m+t\quad\text{and}\quad|I\setminus J|=|J\setminus I|=m-t$$ 
for a certain integer $t>0$. This implies that
 $$|I\cup J|=2m+m-t=3m-t\quad\text{and}\quad\bigl|[n]\setminus (I\cup J)\bigr|\ge 3m-(3m-t)=t.$$
We take $t$-element subsets 
$$A_1\subset I\cap J\quad\text{and}\quad A_2\subset [n]\setminus(I\cup J).$$
%$(m-t)$-element subsets 
%$$A_3\subset I\setminus J\quad\text{and}\quad A_4\subset J\setminus I.$$
Then
%$$A=A_1\cup A_2\cup A_3\cup A_4$$
$$A=A_1\cup A_2\cup (I\setminus J)\cup (J\setminus I)$$
is a $2m$-element subset intersecting each of $I,J$ in an $m$-element subset.
The point $P_A$ is as required.
\end{proof}

Recall that a {\it clique} is a subset in the vertex set of a graph, where any two distinct vertices are connected by an edge. A clique is {\it maximal} if every clique containing it coincides with this clique.

\begin{prop}\label{prop2}
The following assertions are fulfilled:
\begin{enumerate}
\item[(A)] If $n=3m$, then maximal cliques of\/ $\Gamma_m$ are precisely lines of\/ ${\mathcal P}_m$.
\item[(B)] Assume that $n\ge 3m+1$. 
If $n=7$ and $m=2$, then ${\mathcal P}_m$ is a polar space.
In the remaining cases, for any two distinct collinear points $P,P'\in {\mathcal P}_m$
there is a point $Q\in {\mathcal P}_m$ collinear to both $P,P'$ and non-collinear to $P\odot P'$.
\end{enumerate}
\end{prop}

\begin{proof}
Let  $P_I,P_{I'}$ be distinct collinear points of ${\mathcal P}_m$.
Then each of the sets 
$$I_1=I\setminus I',\qquad I_2=I\cap I',\qquad I_3=I'\setminus I$$
contains precisely $m$ elements. 
Take $P_J\in {\mathcal P}_m$ and
$x_i=|J\cap I_i|$, $i\in \{1,2,3\}$. If 
$P_J$ is collinear to both $P_I,P_{I'}$, then
\begin{equation}\label{eq-col}
x_1+x_2=x_2+x_3=m
\end{equation}
which means that $x_1=x_3$.

If $n=3m$, then $J\subset I\cup I'$ and thus $$x_1+x_2+x_3=2m.$$ In this case, \eqref{eq-col}  shows that $x_1=x_3=m$ and $x_2=0$, i.e.\
$J=I\mathbin{\triangle} I'$ and $P_J=P_I\odot P_{I'}$. This implies (A).

Consider the case when $n\ge 3m+1$. 

Suppose that $n=7$ and $m=2$.  
Since $J$ is contained in $I\cup I'$ or intersects it in a $3$-element subset, 
$x_1+x_2+x_3$ is equal to $4$ or $3$ respectively. Note that $x_1, x_2, x_3\le 2$, so one of 
the following possibilities is realized:
\begin{itemize}
\item
$x_1=x_2=x_3=1$, then $P_J$ is collinear to $P_I,P_{I'},P_{I\mathbin{\triangle} I'}$;

\item
precisely one of $x_1,x_2,x_3$ is $2$ and the remaining two both are $1$ or one of them is $0$, then
$P_J$ is collinear to precisely one of $P_I,P_{I'},P_{I\mathbin{\triangle} I'}$;

\item
two of $x_1,x_2,x_3$ are $2$ and the remaining is equal to $0$, then 
$P_J$ coincides with one of $P_I,P_{I'},P_{I\mathbin{\triangle} I'}$.
\end{itemize}
This means that ${\mathcal P}_m$ satisfies one-or-all axiom and thus it is a polar space.

For $m\ne 2$ we take $J$ satisfying 
$$x_1=x_3=m-1\quad\text{and}\quad x_2=1$$
(this is possible, since $[n]\setminus (I\cup I')$ is non-empty).
If $m=2$ and $n\ge 8$, then 
$$\bigl|[n]\setminus (I\cup I')\bigr|\ge 2$$
and we can choose $J$ intersecting $I\cup I'$ precisely in $I\cap I'$, 
i.e.\ $x_1=x_3=0$ and $x_2=2$.
In each of these cases, $P_J$ is collinear to both $P_I,P_{I'}$ and non-collinear to $P_{I\mathbin{\triangle} I'}$.
\end{proof}

Proposition~\ref{prop2} shows that the family of maximal cliques of $\Gamma_m$
coincides with the family of maximal singular subspaces of ${\mathcal P}_m$ if
$n=3m$ or $n=7$ and $m=2$. 
In the remaining  cases, $\Gamma_m$ contains maximal cliques which are not 
maximal singular subspaces of ${\mathcal P}_m$.
By \cite{Ryser}, every maximal clique of $\Gamma_m$ contains no more than $n$ elements;
 %furthermore, elements of this clique are blocks of a certain symmetric block design if the number of elements is equal to $n$.
 %this clique is a symmetric block design if it contains precisely $n$ elements. 
furthermore, if it contains $n$ elements than these elements are blocks of a certain symmetric block design.

By the first part of Proposition~\ref{prop2}, every maximal singular subspace of ${\mathcal P}_m$
is a line if $n=3m$.
Now, we determine maximal singular subspaces of ${\mathcal P}_m$ in the general case.
Since the geometry ${\mathcal P}_m$ is embedded in the projective space ${\mathcal P}(V)$,
every singular subspace of  ${\mathcal P}_m$ corresponds to a certain subspace of $V$.

Every $k$-dimensional subspace $C\subset V$ is 
a {\it binary} $[n,k]$ {\it code}
and vectors of $C$ are called {\it code words} of this code.
The code is {\it non-degenerate} if the restriction of every coordinate functional 
$(x_1,\dots,x_n)\to x_i$ to $C$ is non-zero. 
A {\it generator matrix} of $C$ is 
the $(k\times n)$-matrix whose rows are vectors from a certain basis of $C$.
A code is non-degenerate if and only if its generator matrices do not contain zero columns. 
Two codes in $V$ are {\it equivalent} if there is a coordinate permutation of $V$
transferring one of them to the other.

The {\it Hamming distance} between $e_I,e_J\in V$ is equal to $|I\mathbin{\triangle} J|$
and the Hamming weight of any vector from $V$ is the Hamming distance between this vector and zero.
A code $C\subset V$ is said to be $t$-{\it equidistant} if the Hamming distance between any 
two distinct code words in this code  is equal to $t$, equivalently, all non-zero code words of $C$
are of Hamming weight $t$.
There is a natural one-to-one correspondence between singular subspaces of ${\mathcal P}_m$
and $2m$-equidistant codes of $V$; furthermore,  maximal singular subspaces correspond to  
maximal $2m$-equidistant codes.

Suppose that $n=2^{k}-1$, i.e. $n$ coincides with the number of non-zero vectors in ${\mathbb F}^k$.
In this case, a binary $[n,k]$ code $C\subset V$ is called a {\it binary simplex code of dimension $k$} if 
columns in every generator matrix of $C$ are non-zero and mutually distinct,
in other words, there is a one-to-one correspondence between columns of this matrix
and non-zero vectors of  ${\mathbb F}^k$
(if a certain generator matrix of $C$ satisfies this condition, then it holds for all generator matrices of $C$).
All such codes are equivalent and can be characterized as maximal $(2^{k-1})$-equidistant codes of $V$.

By \cite{Bonis} (see also \cite[Theorem 7.9.5]{HP-book} or \cite{Ward}), a (not necessarily non-degenerate) 
binary $[n,k]$ code $C\subset V$ is $t$-equidistant  if and only if there are natural numbers $s\ge 1$ and $r$ such that
$$n=(2^k-1)s+r,\qquad t=2^{k-1}s$$
and every generator matrix $M$ of $C$ satisfies the following:
\begin{itemize}
\item
$M$ contains precisely $r$ zero-columns;
\item
$M$ contains every non-zero vector of ${\mathbb F}^k$ as a column 
precisely $s$ times; 
\end{itemize}
roughly speaking, $C$ is equivalent to the $s$-fold replication of a binary $k$-dimensional simplex code
with added $r$ zero-coordinates.
This implies that all $t$-equidistant codes of the same dimension in $V$ are equivalent.
The dimension of maximal $t$-equidistant codes in $V$ is the greatest natural $k$ satisfying the following condition:
there is natural $s>0$ such that
$$t=2^{k-1}s\quad\text{and}\quad n\ge (2^k-1)s.$$
Therefore, the (projective) dimension of maximal singular subspaces of ${\mathcal P}_m$
is the greatest natural $k-1$ such that 
$$m=2^{k-2}s\quad\text{and}\quad n\ge (2^k-1)s$$
for a certain natural $s>0$.

\section{Automorphisms}
An {\it automorphism} of a point-line geometry is a bijective transformation of the set of points 
preserving the family of lines in both directions.
For example, every automorphism of a projective space 
is induced by a semilinear automorphism of the corresponding vector space 
(the Fundamental Theorem of Projective Geometry).

Every  linear automorphism  of $V$ preserving ${\mathcal P}_m$ induces an automorphism  of 
the geometry ${\mathcal P}_m$. 
This holds, for example, for every coordinate permutation (monomial linear automorphism).
The examples below show that
there are other linear automorphisms which preserve ${\mathcal P}_m$.

\begin{exmp}\label{exmp-lin1}
Suppose that $n=4m-1$. For every $i\in [n]$ we denote by $l_i$ the linear automorphism of $V$
leaving fixed each $e_j$, $j\ne i$ and transferring $e_i$ to $e_{[n]}$.
Then $l_i(e_I)=e_I$ if $i\not\in I$.
For $I\subset [n]$ containing $i$ we obtain that
$$l_i(e_I)=e_{[n]}+e_{I\setminus\{i\}}=e_{[n]\setminus (I\setminus \{i\})};$$
if $|I|=2m$, then $$\bigl|[n]\setminus (I\setminus \{i\})\bigr|=4m-1-(2m-1)=2m.$$
Therefore, $l_i$ preserves ${\mathcal P}_m$. 
\end{exmp}

\begin{rem}
Suppose that $n=4m-1=3$. If $i\in\{1,2,3\}$ and $j,s$
are the remaining two elements of $\{1,2,3\}$, then 
the automorphism of ${\mathcal P}_1$ induced by $l_i$ 
 is also induced by the transposition of $j$-th and $s$-th coordinates. 
\end{rem}

\begin{exmp}\label{exmp-lin3}
%%As in the previous example, we 
Assume that $n=4m$.
For distinct $i,j\in [n]$ denote by $s_{ij}$ and $s'_{ij}$ the linear automorphisms of $V$ which leave each
$e_t$, $t\ne i,j$ fixed and satisfy the following conditions: 
\begin{align*}
s_{ij}(e_i) &= e_{[n]\setminus \{i\}}, & s_{ij}(e_j) &= e_{[n]\setminus \{j\}},\\
s'_{ij}(e_i) &= e_{[n]\setminus \{j\}}, & s'_{ij}(e_j) &= e_{[n]\setminus \{i\}}.
\end{align*}
Each of them  leaves $e_I$ fixed if $i,j\not\in I$.
If $i,j\in I$, then
$$s_{ij}(e_I)=s'_{ij}(e_I)=e_{[n]\setminus\{i\}}+e_{[n]\setminus \{j\}}+e_{I\setminus\{i,j\}}=
e_{\{i,j\}}+e_{I\setminus\{i,j\}}=e_I.$$
Suppose that $I$ is a $2m$-element subset of $[n]$ containing only one of $i,j$, say $i$.
Then
$$s_{ij}(e_I)=e_{[n]\setminus \{i\}}+e_{I\setminus\{i\}}=e_{[n]\setminus I}$$
and
$$s'_{ij}(e_I)=e_{[n]\setminus \{j\}}+e_{I\setminus\{i\}}=e_{([n]\setminus \{j\})\setminus(I\setminus\{i\})}$$
It is clear that $\bigl|[n]\setminus I\bigr|=2m$. 
Since $j\not\in I$, $I\setminus\{i\}$ is a $(2m-1)$-subset of the $(4m-1)$-element set  $[n]\setminus \{j\}$
which implies that
$$\bigl|([n]\setminus \{j\})\setminus(I\setminus\{i\})\bigr|=2m.$$
So, $s_{ij}$ and $s'_{ij}$ both preserve ${\mathcal P}_m$.
\end{exmp}

\begin{rem}
Suppose that $n=4m=4$. If $i,j\in \{1,2,3,4\}$  and $s,t$ are the remaining two elements of 
$ \{1,2,3,4\}$, then the automorphism of ${\mathcal P}_1$ induced by $s_{ij}$ is also induced 
by the transposition of $i$-th, $j$-th coordinates and $s$-th, $t$-th coordinates. 
Similarly, $s'_{ij}$ induces the automorphism of ${\mathcal P}_1$ which is also induced by 
the transposition of $s$-th, $t$-th coordinates.
\end{rem}

Our main result is the following.

\begin{theorem}\label{theorem-aut}
Every automorphism of the geometry  ${\mathcal P}_m$ is induced by 
a coordinate permutation of $V$ or is the composition of the automorphism induced by 
a coordinate permutation and the automorphism induced by one of the linear automorphisms 
considered in Examples  \ref{exmp-lin1}, \ref{exmp-lin3}.
\end{theorem}

\begin{rem}\label{rem-case1}
Recall that $\Gamma_1$ is the Johnson graph $J(n,2)$.
If $n\ne 4$, then every automorphism of this graph is induced by a permutation on $[n]$
which can be identified with a coordinate permutation of $V$. 
For $n=4$ the same holds only in the case when an automorphism of $\Gamma_1$ preserves  the types of maximal cliques, i.e. stars go to stars and tops go to tops (see Appendix). 
Since tops of $\Gamma_1$ correspond to lines of ${\mathcal P}_1$,
automorphisms of the geometry ${\mathcal P}_1$ are precisely such automorphisms of $\Gamma_1$.
\end{rem}

\begin{rem}\label{rem-polar}
By Proposition \ref{prop2}, ${\mathcal P}_m$ is a polar space
if $n=7$ and $m=2$. 
It is well-known that every automorphism of the collinearity graph of a polar space 
is an automorphism of this polar space.
\end{rem}

\begin{exmp}
Suppose that $n=4m$.
Since
$$([n]\setminus I)\mathbin{\triangle}([n]\setminus J)=I\mathbin{\triangle} J$$
for any $I,J\subset [n]$,
the map transferring every $P_I\in {\mathcal P}_m$ to $P_{[n]\setminus I}$
is an automorphism of $\Gamma_m$ which is not an automorphism of ${\mathcal P}_m$.
\end{exmp}

Is it possible to determine automorphisms of the graph $\Gamma_m$? 
In particular, is it true that every automorphism of $\Gamma_m$ is an automorphism of ${\mathcal P}_m$ if $n\ne 4m$?

For every subset ${\mathcal X}$ in the point set of a geometry we denote by ${\mathcal X}^c$
the set of points collinear to all points of ${\mathcal X}$.
If  $p,p'$ are distinct collinear points of a polar space, then $\{p,p'\}^{cc}$ is the line containing these points.
In the case when $n>3m$ and $m\ne 7$ or $m\ne2$, 
Proposition \ref{prop2} shows that for any distinct collinear $P,P'\in {\mathcal P}_m$
the set $\{P,P'\}^{cc}$ does not contain $P\odot P'$. 
Using the form of generator matrices for equidistant codes, we can show that 
the intersection of all maximal singular subspaces of ${\mathcal P}_{m}$ containing 
the line $\langle P,P'\rangle$ coincides with this line which implies that $\{P,P'\}^{cc}=\{P,P'\}$.
So, lines of ${\mathcal P_m}$ cannot be characterized in terms of 
the binary collinearity relation in a standard way.
Therefore, we need more information on maximal cliques of $\Gamma_m$ which are not 
maximal singular subspaces of ${\mathcal P}_m$. At this moment, we can say only that 
they contain no more than $n$ elements. 

\begin{rem}\label{rem-psts}
A subgeometry of a certain point-line geometry is called 
{\it smooth} \cite{Cameron} or {\it rigid} \cite{Pankov-book}
if every automorphism of this subgeometry can be extended to an automorphism of the geometry. 
By \cite{Cameron}, every partial Steiner triple systems ($\psts$)  
 of order $u\ge 4$ is  rigidly embeddable in a Steiner triple systems ($\sts$) of order $2^{u-1}-1$
and there is a $\psts$ of order $u$ which is not rigidly  embeddable in any smaller $\sts$.
In the next section, we show that the geometry ${\mathcal P}_m$ is rigidly embedded in ${\mathcal P}(S)$, where $S$ is a certain hyperplane of $V$. 
Note that ${\mathcal P}_m$ is a $\psts$ of order $\binom{n}{2m}$ and ${\mathcal P}(S)$ is a $\sts$ of order $2^{n-1}-1$. 
Since $m\ge 1$, we have 
$$n=\binom{n}{1}<\binom{n}{2m},$$
and, consequently,
$$2^{n-1}-1<2^{\binom{n}{2m}-1}-1$$
which means that our $\psts$ is rigidly embedded in a $\sts$ of a significantly smaller order than it is stated in \cite{Cameron}.
\end{rem}

\section{Proof of  Theorem \ref{theorem-aut}}
For $m=1$ the statement is obvious (Remark \ref{rem-case1}) and we assume that $m\ge 2$.

\subsection{First step}\label{subs:firststep}
Let $S$ be the hyperplane of $V$ formed by all vectors $(x_1,\dots,x_n)$ satisfying 
$$x_1+\dots+x_n=0.$$
Then $P_I\in {\mathcal P}(V)$ is contained in $S$ if and only if $|I|$ is even.
It is clear that $P\odot P'$ belongs to ${\mathcal P}(S)$ for any distinct $P,P'\in {\mathcal P}_m$.
Denote by $\overline{{\mathcal P}}_m$ the set of all points $Q\in {\mathcal P}(S)$
such that $Q=P\odot P'$ for some $P,P'\in {\mathcal P}_m$.
If $n\ge 4m$, then $[n]$ contains disjoint $2m$-element subsets and
$$\overline{{\mathcal P}}_m=\bigcup^{2m}_{i=1}{\mathcal P}_i.$$
If $n<4m$, then for $2m$-element subsets $I,J\subset [n]$ the minimum value of $|I\cap J|$ is
$4m-n$ and, consequently, the maximum value of $|I\setminus J|=|J\setminus I|$ is $n-2m$ which implies that 
$$\overline{{\mathcal P}}_m=\bigcup^{n-2m}_{i=1}{\mathcal P}_i.$$
Therefore, $\overline{{\mathcal P}}_m$ is the union of all ${\mathcal P}_i$ such that 
$$i\le \min\{2m, n-2m\}.$$
A direct verification shows that $\overline{{\mathcal P}}_m={\mathcal P}(S)$ only in the case when $n=4m+\epsilon$
with $\epsilon\in \{-1,0,1\}$.

\begin{rem}
Every non-zero vector of $S$ can be presented as the sum of some vectors of Hamming weight $2m$.
Thus ${\mathcal P}(S)$ is the smallest projective space containing ${\mathcal P}_m$.
\end{rem}

Let $P$ be a point belonging to $\overline{{\mathcal P}}_m\setminus {\mathcal P}_m$.
Consider the family ${\mathcal L}_P$ formed by all lines of ${\mathcal P}(S)$ passing through $P$
and containing two points from ${\mathcal P}_m$. 
On this family we define the relation $\sim$ as follows: 
for mutually distinct $P_I,P_J,P_{I'},P_{J'}\in {\mathcal P}_m$
satisfying 
$$P_I\odot P_J=P_{I'} \odot P_{J'}=P$$
we write
$$\langle P_I,P_J\rangle\sim \langle P_{I'}, P_{J'}\rangle$$
if at least one of the points 
\begin{equation}\label{eq-2ponts}
P_I\odot P_{I'}=P_J \odot P_{J'},\qquad P_I\odot P_{J'}=P_J \odot P_{I'}
\end{equation}
belongs to ${\mathcal P}_m$; in other words, the lines 
$\langle P_I,P_J\rangle,\langle P_{I'}, P_{J'}\rangle$ can be extended  to a Pasch configuration
such that the remaining two lines are contained in ${\mathcal P}_m$, see Figure~\ref{fig:veblenplus}.
If the points \eqref{eq-2ponts} both belong to ${\mathcal P}_m$, 
then $\langle P_I,P_J\rangle,\langle P_{I'}, P_{J'}\rangle$ 
span a Fano plane, where all points, except $P$, belong to ${\mathcal P}_m$.

\begin{figure}[ht]
\begin{center}
\includegraphics[scale=0.65]{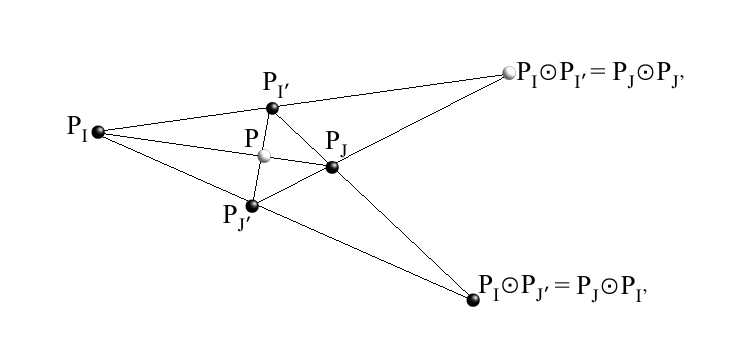}
\caption{Two Pasch configurations with the lines $\langle P_{I}, P_{J}\rangle$ and $\langle P_{I'}, P_{J'}\rangle$.}
\label{fig:veblenplus}
\end{center}
\end{figure}  

\begin{prop}\label{prop-conn}
If $n\ge 3m+1$, then 
for any distinct lines $l,l'\in {\mathcal L}_P$ there is a sequence of lines
$$l=l_0\sim l_1\sim \dots\sim l_m=l'$$
belonging to ${\mathcal L}_P$.
If $n=3m$, then such a sequence exists for any distinct lines $l,l'\in {\mathcal L}_P$ only 
in the case when $P\in {\mathcal P}_1$.
\end{prop}

\begin{proof}
Let $P=P_{N_1}$ and $|N_1|=2i$. Then $N_2=[n]\setminus N_1$
contains precisely $n-2i$ elements.
If $P_I,P_J\in {\mathcal P}_m$ and $P_I\odot P_J=P$,
then $I\mathbin{\triangle} J=N_1$ which implies that
$$I=I_1\cup I_2\quad\text{and}\quad J=(N_1\setminus I_1)\cup I_2,$$
where $I_1$ is an $i$-element subset of $N_1$ and $I_2$ is a $(2m-i)$-element subset of $N_2$.
Similarly, for other $P_{I'},P_{J'}\in {\mathcal P}_m$ satisfying $P_{I'}\odot P_{J'}=P$ we have
$$I'=I'_1\cup I'_2\quad\text{and}\quad J'=(N_1\setminus I'_1)\cup I'_2,$$
where $I'_1$ is an $i$-element subset of $N_1$ and $I'_2$ is a $(2m-i)$-element subset of $N_2$.
The following three conditions are equivalent:
\begin{itemize}
\item $|I\cap I'|=|I_1\cap I'_1|+|I_2\cap I'_2|=m$;
\item $|J\cap J'|=\bigl|(N_1\setminus I_1)\cap (N_1\setminus I'_1)\bigr|+|I_2\cap I'_2|=m$;
\item $P_I\odot P_{I'}=P_J\odot P_{J'}$ belongs to ${\mathcal P}_m$.
\end{itemize}
For the same reason, the following three conditions also are equivalent:
\begin{itemize}
\item $|I\cap J'|=|I_1\cap (N_1\setminus I'_1)|+|I_2\cap I'_2|=m$;
\item $|J\cap I'|=\bigl|(N_1\setminus I_1)\cap I'_1\bigr|+|I_2\cap I'_2|=m$;
\item$P_I\odot P_{J'}=P_J\odot P_{I'}$ belongs to ${\mathcal P}_m$.
\end{itemize}

{\it The case $n\ge 3m+1$.}
Suppose that $i<m$. 
There are pairs of $(2m-i)$-element subsets of $N_2$ whose intersection contains precisely $m-1$ elements and,
consequently, pairs of such subsets with an $m$-element intersection also exist. 
Indeed,
$$(m-i+1)+(m-1)+(m-i+1)=3m-2i+1\le n-2i.$$
First, we consider the case when $I_1=I'_1$ and $I_2\ne I'_2$.
By Lemma \ref{lemma-John} (see Appendix),
there is a sequence of $(2m-i)$-element subsets 
$$I_2=I^0_2,I^1_2,\dots,I^t_2=I'_2\subset N_2$$
such that 
$$|I^{u-1}_2\cap I^{u}_2|=m$$
for every $u\in\{1,\dots,t\}$.
If 
$$I^u=I_1\cup I^u_2\quad\text{and}\quad J^u=(N_1\setminus I_1)\cup I^u_2,$$
then
$$|I^{u-1}\cap J^u|=|I_1\cap (N_1\setminus I_1)|+|I^{u-1}_2\cap I^{u}_2|=m$$
%$$|I^{u-1}\cap J^u|=|J^{u-1}\cap I^u|=m$$
which implies that 
$$\langle P_I,P_J\rangle =\langle P_{I^0},P_{J^0}\rangle \sim\langle P_{I^1},P_{J^1}\rangle \sim\dots \sim \langle P_{I^t},P_{J^t}\rangle=\langle P_{I'},P_{J'}\rangle.$$
Now,  we assume that  $I_1\ne I'_1$.
Lemma \ref{lemma-John} implies the existence of  a sequence of $2i$-element subsets 
$$I_1=I^0_1,I^1_1,\dots,I^t_1=I'_1\subset N_1$$
such that 
$$|I^{u-1}_1\cap I^{u}_1|=1$$
for every $u\in\{1,\dots,t\}$.
We take any sequence of $(2m-i)$-element subsets
$$I_2=I^0_2,I^1_2,\dots,I^t_2\subset N_2$$
such that 
$$|I^{u-1}_2\cap I^{u}_2|=m-1$$
for every  $u\in\{1,\dots,t\}$
(it was noted above that pairs of $(2m-i)$-element subsets with such intersections exist).
If
$$I^u=I^u_1\cup I^u_2\quad\text{and}\quad J^u=(N_1\setminus I^u_1)\cup I^u_2,$$
then for every  $u\in\{1,\dots,t\}$ we have
$$|I^{u-1}\cap I^u|=|I^{u-1}_1\cap I^{u}_1|+|I^{u-1}_2\cap I^{u}_2|=1+(m-1)=m$$
which means that 
$$\langle P_I,P_J\rangle =\langle P_{I^0},P_{J^0}\rangle \sim\langle P_{I^1},P_{J^1}\rangle \sim\dots \sim \langle P_{I^t},P_{J^t}\rangle.$$
Since 
$$I^t=I'_1\cup I^t_2\quad\text{and}\quad J^t=(N_1\setminus I'_1)\cup I^t_2,$$
the lines $\langle P_{I^t},P_{J^t}\rangle$ and 
$\langle P_{I'},P_{J'}\rangle$ are connected by a sequence of lines 
from ${\mathcal L}_P$ via the relation $\sim$.

Suppose that $i>m$. 
Then $0\le 2m-i<m$. Note that $N_2=\emptyset$ if $i=2m$.
In the case when $I_2\ne I'_2$, we take any $i$-element subset $I''_1\subset N_1$ such that
$$|I'_1\cap I''_1|=m-|I_2\cap I'_2|$$
and define
$$I''=I''_1\cup I_2,\;\;\; J''=(N_1\setminus I''_1)\cup I_2;$$
since 
$$|I'\cap I''|=|I'_1\cap I''_1|+|I_2\cap I'_2|=m,$$
we obtain that
$$\langle P_{I'},P_{J'}\rangle \sim\langle P_{I''},P_{J''}\rangle.$$
So, we can assume that $I_2=I'_2$. Consider a sequence of $i$-element subsets
$$I_1=I^0_1,I^1_1,\dots,I^t_1=I'_1\subset N_1$$
such that 
$$|I^{u-1}_1\cap I^{u}_1|=i-m$$
for every $u\in \{1,\dots,t\}$ (Lemma \ref{lemma-John}).
If
$$I^u=I^u_1\cup I_2\quad\text{and}\quad J^u=(N_1\setminus I^u_1)\cup I_2,$$
then
$$|I^{u-1}\cap I^u|=|I^{u-1}_1\cap I^{u}_1|+|I_2|=i-m+2m-i=m$$
which means that 
$$\langle P_I,P_J\rangle =\langle P_{I^0},P_{J^0}\rangle \sim\langle P_{I^1},P_{J^1}\rangle \sim\dots \sim \langle P_{I^t},P_{J^t}\rangle=\langle P_{I'},P_{J'}\rangle.$$

{\it The case $n=3m$.}
Then $i<m$. The equality
$$2(2m-i)-m=3m-2i=n-2i$$
shows that 
for $(2m-i)$-element subsets $A,B\subset N_2$
the minimum value of $|A\cap B|$ is $m$.
Therefore, 
$\langle P_I,P_J\rangle \sim\langle P_{I'},P_{J'}\rangle$ if and only if 
$|I_2\cap I'_2|=m$ and $I_1$ coincides with $I'_1$ or $N_1\setminus I'_1$.
This means that $\langle P_I,P_J\rangle$ and $\langle P_{I'},P_{J'}\rangle$ are connected 
by a sequence of lines from ${\mathcal L}_P$ (via the relation $\sim$) only in the case when 
$I_1$ coincides with $I'_1$ or $N_1\setminus I'_1$. The latter is always true only for $i=1$.
\end{proof}

Let  $P_X\in{\overline {\mathcal P}}_m$.
Then $|X|=2i$ and $1\le i\le  \min\{2m, n-2m\}$.
We count  all the lines passing through $P_X$ whose two remaining points
belong to ${\mathcal P}_m$, i.e. all the lines $\langle P_I,P_J\rangle$ such that 
$$|I|=|J|=2m\quad\text{and}\quad I\mathbin{\triangle} J=X.$$
If we choose an $i$-element subset $A\subset X$ and 
a $(2m-i)$-element subset $B\subset [n]\setminus X$,
then 
$$I=A\cup B\quad\text{and}\quad J=(X\setminus A)\cup B$$
are as required.
So, the number of such lines is
$$\lambda_{i} = \frac{1}{2}\binom{2i}{i}\binom{n-2i}{2m-i}.$$
The inequality $i\le \min\{2m, n-2m\}$ guarantees that $2m-i\le n-2i$ and, consequently, 
$n=2i$ implies that $i=2m$ (recall that $\binom{0}{0}=1$).

\begin{lemma}\label{lemma-lambda}
Let $i$ be an integer satisfying $2\le i\le \min\{2m, n-2m\}$.
Then $\lambda_i\ne \lambda_1$ except in the following cases:
\begin{itemize}
\item
$\lambda_1=\lambda_{2m-1}$ if $n=4m-1$ or $n=4m$;
\item
$\lambda_1=\lambda_{2m}$ if $n=4m+1$.
\end{itemize}
\end{lemma}

\begin{proof}
First of all, we observe that 
$$\lambda_1=\binom{n-2}{2m-1}$$
is the number of all $2m$-element subsets of $[n-1]$ containing fixed $j\in [n-1]$, say $j=1$. 
If $n\ge 4m$, then 
$$\lambda_{2m}=\frac{1}{2}\binom{4m}{2m}=\frac{1}{2}\Bigg[\binom{4m-1}{2m-1}+\binom{4m-1}{2m}\Bigg]=
\binom{4m-1}{2m-1}$$
is the number of all $2m$-element subsets of $[4m]$ containing $1$ which implies the following:
\begin{itemize}
\item $\lambda_{2m}>\lambda_1$ if $n=4m$;
\item $\lambda_{2m}=\lambda_1$ if $n=4m+1$;
\item $\lambda_{2m}<\lambda_1$ if $n>4m+1$.
\end{itemize}
From this moment, we assume that $1<i<2m$.
Then $1\le 2m-i\le n-2i$.
If $2m-i<n-2i$, then
\begin{multline*}
\lambda_i=\frac{1}{2}\binom{2i}{i}\binom{n-2i}{2m-i}=
\binom{2i-1}{i}\Bigg[\binom{n-2i-1}{2m-i}+\binom{n-2i-1}{2m-i-1}\Bigg] = \\
=\Bigg[\binom{2i-1}{i-1}\binom{n-2i-1}{2m-i}+\binom{2i-1}{i}\binom{n-2i-1}{2m-i-1}\Bigg];
\end{multline*}
in the case when $2m-i=n-2i$, we have
$$\lambda_i=\binom{2i-1}{i}=\binom{2i-1}{i}\binom{n-2i-1}{2m-i-1}.$$
Let $I=[2i]$. 
If $2m-i<n-2i$, then $\lambda_i$ is the number of all $2m$-element subsets $Y\subset [n-1]$ containing $1$
and such that 
\begin{equation}\label{eq:special}
|I\cap Y|=i\quad\text{or}\quad |I\cap Y|=i+1.
\end{equation}
If $2m-i=n-2i$, then 
$\lambda_i$ is the number of all $2m$-element subsets 
$Y\subset [n-1]$ containing $1$ and 
satisfying the second equality only.
This means that $\lambda_i\le \lambda_1$.

The equality $\lambda_i= \lambda_1$ implies that one of the following possibilities is realized:
\begin{itemize}
\item $2m-i=n-2i$ and any  $2m$-element subset 
$Y\subset [n-1]$ containing $1$
satisfies the second equality of \eqref{eq:special};
\item $2m-i<n-2i$ and any  $2m$-element subset 
$Y\subset [n-1]$ containing $1$
satisfies \eqref{eq:special}.
\end{itemize}
In the first case, $2i=n-1$ (otherwise, %there are more than one possibility  for
$|I\cap Y|$ takes more than one value)
and, consequently, 
$$i+1=|I\cap Y|=|Y|=2m$$
which means that $i=2m-1$ and $n=4m-1$.
In the second case, we have $2i=n-2$ 
(otherwise, $|I\cap Y|$ takes precisely one or more than one value); %we obtain precisely one possibility  for $|I\cap X|$ or more than two possibilities);
then 
$$|I\cap Y|=2m-1\quad\text{or}\quad |I\cap Y|=|Y|=2m$$
which shows that $i=2m-1$ and $n=4m$.
\end{proof}

\begin{lemma}\label{lemma-def}
Suppose that $P,Q,P',Q'$ are mutually distinct points of  ${\mathcal P}_m$ such that
$P\odot Q=P'\odot Q'$ and one of the following possibilities is realized:
\begin{itemize}
\item $n\ge 3m+1$,
\item $n=3m$ and the point $P\odot Q=P'\odot Q'$ belongs to ${\mathcal P}_1$.
\end{itemize}
Then 
$$f(P)\odot f(Q)=f(P')\odot f(Q').$$
\end{lemma}

\begin{proof}
If  $\langle P,Q\rangle \sim \langle P',Q'\rangle$, 
then the lines $\langle P,Q\rangle$ and $\langle P',Q'\rangle$ are contained
in a Pasch configuration, where the remaining two lines $l_1,l_2$ are formed by points of ${\mathcal P}_m$.
The points $P,Q,P',Q'$ belong to $l_1\cup l_2$. 
Since $f$ is an automorphism of ${\mathcal P}_m$, 
the lines $\langle f(P),f(Q)\rangle$ and $\langle f(P'),f(Q')\rangle$ are in 
a Pasch configuration whose remaining two lines are $f(l_1),f(l_2)$ which implies the required equality.
In the general case, we use Proposition \ref{prop-conn}.
\end{proof}

If $n\ge 3m+1$, then
Lemma \ref{lemma-def} shows that $f$ can be extended to a transformation of $\overline{{\mathcal P}}_m$
as follows: for every point $S\in \overline{{\mathcal P}}_m\setminus {\mathcal P}_m$ we take 
points $P,Q\in {\mathcal P}_m$  satisfying $P\odot Q=S$ and define
$$f(S)=f(P)\odot f(Q).$$
This extension is bijective and preserves in both directions 
the family of lines containing at least two points from ${\mathcal P}_m$.
By Lemma \ref{lemma-lambda}, we obtain that $f({\mathcal P}_1)={\mathcal P}_1$ if 
$n\ne 4m+\epsilon$ with $\epsilon\in \{-1,0,1\}$.

In the case when $n=3m$, we extend $f$ to an injective map of ${\mathcal P}_1\cup {\mathcal P}_m$
to $\overline{{\mathcal P}}_m$.
If  $P\in {\mathcal P}_1$, then there are $\lambda_1$ lines passing through $f(P)$ and containing two points from 
${\mathcal P}_m$.
Since $\min\{2m, n-2m\}=m$ and $\lambda_1>\lambda_i$ if $1<i\le m$ (see the proof of Lemma \ref{lemma-lambda}),
$f(P)$ belongs to ${\mathcal P}_1$.
So, our extension is a bijective transformation of ${\mathcal P}_1\cup {\mathcal P}_m$.
It preservers in both directions 
the family of lines contained in ${\mathcal P}_1\cup {\mathcal P}_m$ and containing
at least two points from ${\mathcal P}_m$.

\subsection{Second step}
We will need some generalizations of Proposition \ref{prop-coll-graph}.

\begin{lemma}\label{lemma-tex}
If $P_I,P_J\in {\overline {\mathcal P}}_m$ are of the same Hamming weight not greater than $2m$ and the Hamming weight of $P_I\odot P_J$ is also  not greater than $2m$, 
then there is $P_A\in {\mathcal P}_m$ such that $P_A\odot P_I$ and $P_A\odot P_J$ belong to ${\mathcal P}_m$.
\end{lemma}

\begin{proof}
Let $|I|=|J|=2i\le 2m$. 
We need to construct a $2m$-element subset $A\subset [n]$ such that
$$|A\cap I|=|A\cap J|=i.$$
Indeed, the latter implies that $|A\mathbin{\triangle} I|=|A\mathbin{\triangle} J|=m$.

Suppose that $|I\setminus J|=|J\setminus I|$ is equal to $t$. Then
$$|I\cap J|=2i-t,\quad |I\cup J|=2i+t,\quad
\bigl|[n]\setminus (I\cup J)\bigr|\ge 3m-2i-t.$$
Since $|I\mathbin{\triangle} J|=2t\le m$, we obtain that $t\le m$.

In the case when $i\le t$, we have 
$$|I\setminus J|=|J\setminus I|\ge i.$$
The inequality $t\le m$ shows that
$$\bigl|[n]\setminus (I\cup J)\bigr|\ge 3m-2i-t\ge 2m-2i.$$
We take $i$-element subsets 
$$A_1\subset I\setminus J,\qquad A_2\subset J\setminus I$$
and a $(2m-2i)$-element subset $$A_3\subset [n]\setminus (I\cup J).$$
Then $A=A_1\cup A_2\cup A_3$ is as required.

Suppose that $i>t$ and choose an $(i-t)$-element subset $A_1\subset I\cap J$.
Since $i\le m$, we obtain that
$$|[n]\setminus (I\cup J)|\ge 3m-2i-t\ge 2m-i-t,$$
i.e. there is a $(2m -i-t)$-element subset $$A_2\subset [n]\setminus (I\cup J).$$
Then 
$$A=(I\setminus J)\cup (J\setminus I)\cup A_1\cup A_2$$
is as required.
\end{proof}

For the case when $n=4m+\epsilon$, $\epsilon\in \{-1,0,1\}$
we establish a more strong statement.

\begin{lemma}\label{lemma-gen-coll-graph}
Suppose that $n=4m+\epsilon$, $\epsilon\in \{-1,0,1\}$.
Then for any $P_I,P_J\in {\overline {\mathcal P}}_m$ there is $P_A\in {\mathcal P}_m$ such that $P_A\odot P_I$ and $P_A\odot P_J$ belong to ${\mathcal P}_m$.
\end{lemma}

\begin{proof}
Let $|I|=2i$, $|J|=2j$ and $i\le j$. 
As in the proof of Lemma \ref{lemma-tex},
we construct a $2m$-element subset $A\subset [n]$ such that
$$|A\cap I|=i\quad\text{and}\quad|A\cap J|=j.$$

(A). Suppose that $i+j\le 2m$. So, there is a natural number $u\ge 0$ such that
$$i+j+u=2m.$$
Then 
\begin{equation}\label{eq:caseA}
2i+2j+2u+\epsilon=n
\end{equation}
which implies that 
$$\bigl|[n]\setminus(I\cup J)\bigr|\ge 2u+\epsilon\ge u\quad\text{if}\quad u>0$$
and $[n]\setminus(I\cup J)$ can be empty if $u=0$.

If $|I\cap J|\le i$, then
$$|I\setminus J|\ge i\quad\text{and}\quad |J\setminus I|=2j-|I\cap J|\ge 2j-i\ge j.$$
We take an $i$-element subset $A_1\subset I\setminus J$, a $j$-element subset $A_2\subset J\setminus I$
and a $u$-element subset $A_3\subset [n]\setminus (I\cup J)$.
Then
\begin{equation}\label{eq2-1}
A=A_1\cup A_2\cup A_3
\end{equation}
 is as requited.

If $|I\cap J|>i$, then 
$$|I\setminus J|\le i-1\quad\text{and}\quad |I\cup J|=2j+|I\setminus J|\le 2j+i-1$$
which, by \eqref{eq:caseA}, implies that 
$$\bigl|[n]\setminus(I\cup J)\bigr|\ge n-(2j+i-1)=2u+i\ge u+i.$$
Since $|I\cap J|\le 2i$, we have
$$|J\setminus I|=2j-|I\cap J|\ge 2j-2i\ge j-i.$$
Therefore, there are an $i$-element subset $A_1\subset I\cap J$, a $(j-i)$-element subset $A_2\subset J\setminus I$
and $(u+i)$-element subset $A_3\subset [n]\setminus (I\cup J)$.
As above, we take \eqref{eq2-1}.

(B). Now, we assume that $i+j> 2m$, i.e.  
$$i+j-u=2m$$
for a certain positive integer $u$. 
Then 
\begin{equation}\label{eq:caseB}
2i+2j-2u+\epsilon=n
\end{equation}
which means that
$$|I\cap J|\ge 2u-\epsilon\ge 2u-1.$$
Therefore, each of $2i,2j$ is greater than $2u-1$ and, consequently,
each of $i,j$ is not less than $u$.

Suppose that $|I\cap J|\le i$. 
Since 
$$|I\setminus J|\ge i,\;|J\setminus I|\ge j,$$
$$|I\cap J|\ge 2u-1\ge u,$$
we can choose a $u$-element subset $A_1\subset I\cap J$,
an $(i-u)$-element subset $A_2\subset I\setminus J$ and a $(j-u)$-element subset $A_3\subset J\setminus I$.
Then \eqref{eq2-1} is as required.

In the case when $|I\cap J|> i$, there is a positive integer $u'$ such that $|I\cap J|=i+u'$.
It is clear that $u'\le i$ and
 the following two possibilities appear:
\begin{itemize}
\item $u\ge u'$,
\item $u<u'.$
\end{itemize}
In the first case, we have
$$|I\cap J|\ge 2u-1\ge u,$$
$$|I\setminus J|=2i-|I\cap J|=2i- i-u'=i-u'\ge i-u\ge 0,$$
$$|J\setminus I|=2j-|I\cap J|=2j-i-u'\ge j-u'\ge j-u\ge 0.$$
As above, we choose 
a $u$-element subset $A_1\subset I\cap J$, an $(i-u)$-element subset $A_2\subset I\setminus J$,
a $(j-u)$-element subset $A_3\subset J\setminus I$
and take \eqref{eq2-1}.

In the second case,
we choose a $u'$-element subset $A_1\subset I\cap J$. 
The inequality 
$$|J\setminus I|=2j-i-u'\ge j-u'$$
guarantees the existence of a $(j-u')$-element subset $A_2\subset J\setminus I$. 
Since $|I\setminus J|=i-u'$, the set
$$(I\setminus J)\cup A_1\cup A_2$$
intersects $I$ and $J$ in an $i$-element and $j$-element subset, respectively.
However, this set contains $i+j-u'$ elements only.
On the other hand,
$$|I\cup J |=2j+i-u'.$$
Applying \eqref{eq:caseB} and $u'\ge u+1$ we obtain that
$$\bigl|[n]\setminus(I\cup J)\bigr|= n-2j-i+u'= i+u'-2u+\epsilon \ge i-u.$$ 
Since  $u'\le i$, we have
$$\bigl|[n]\setminus(I\cup J)\bigr|\ge u'-u.$$
For any $(u'-u)$-element subset $A_3\subset [n]\setminus (I\cup J)$
the set
$$A=(I\setminus J)\cup A_1\cup A_2\cup A_3$$
is as required.
\end{proof}

The union of three lines on a Fano plane  is said to be a {\it triangle} 
if the pairwise intersections of these lines are mutually distinct.
A triangle contains precisely three points distinct from the vertices (the intersecting points), these points form a line. 
If a line $l$ contains two points $P_I,P_J\in {\overline {\mathcal P}}_m$ and 
there is a point $P_{A}\in {\mathcal P}_m$ such that $P_A\odot P_I$ and $P_A\odot P_J$ belong to ${\mathcal P}_m$
(this happens, for example, when $P_I,P_J$ satisfy the conditions of
Lemma~\ref{lemma-tex} or $n=4m+\epsilon$ with $\epsilon\in \{-1,0,1\}$),
then the lines 
\begin{equation}\label{eq-3lines}
\langle P_A,P_I\rangle,\quad \langle P_A,P_J\rangle,\quad \langle P_A\odot P_I,P_A\odot P_J\rangle
\end{equation}  
form a triangle whose vertices are points from ${\mathcal P}_m$
and the line $l$ can be characterized as the set of points of this triangle distinct from the vertices.

Suppose that $n\ge 3m+1$.
It was established at the end of the previous subsection that $f$ can be extended to a bijective transformation of ${\overline {\mathcal P}}_m$ which we also denote by $f$.
This extension preservers in both directions the family of all lines contained in ${\overline {\mathcal P}}_m$
and containing at least two points from ${\mathcal P}_m$.
If  $n=4m+\epsilon$ with $\epsilon\in \{-1,0,1\}$,
then $\overline{{\mathcal P}}_m={\mathcal P}(S)$.
We use the above observation together with  Lemma \ref{lemma-gen-coll-graph} to prove the following.

\begin{lemma}\label{lemma-epsilon-case}
If  $n=4m+\epsilon$ with $\epsilon\in \{-1,0,1\}$, then $f$ is induced by a linear automorphism of $S$.
\end{lemma}

\begin{proof}
By Lemma \ref{lemma-gen-coll-graph},
for any $P_I,P_J\in {\mathcal P}(S)$  
there is a point $P_{A}\in {\mathcal P}_m$ such that $P_A\odot P_I$ and $P_A\odot P_J$ belong to ${\mathcal P}_m$.
The lines \eqref{eq-3lines} form a triangle whose vertices are points from ${\mathcal P}_m$
and the line $\langle P_I,P_J\rangle$
is the set of points of this triangle distinct from the vertices.
Since $f$ preservers the family of all lines of ${\mathcal P}(S)$
containing at least two points from ${\mathcal P}_m$, 
it transfers \eqref{eq-3lines} to a triangle and $f(\langle P_I,P_J\rangle)$ is 
the set of points of this triangle distinct from the vertices. 
Therefore, $f(\langle P_I,P_J\rangle)$ is a line.
Applying the same arguments to $f^{-1}$ we establish that $f$
preservers the family of all lines of ${\mathcal P}(S)$ in both directions.
The Fundamental Theorem of Projective Geometry gives the claim.
\end{proof}

In the general case, 
we observe that points $P_I,P_J\in {\overline {\mathcal P}}_m$ satisfy the conditions of Lemma \ref{lemma-tex}
if and only if the line $\langle P_I,P_J\rangle$ is contained in 
$\cup^m_{i=1}{\mathcal P}_i$.
As in the proof of Lemma \ref{lemma-epsilon-case}, we show that $f$ sends every such line to a line.

\begin{lemma}\label{lemma-p1}
If $n\ge 3m+1$ and $f({\mathcal P}_1)={\mathcal P}_1$, then $f$
is induced by a coordinate permutation of $V$.
\end{lemma}

\begin{proof}
The assumption $f({\mathcal P}_1)={\mathcal P}_1$ guarantees that $f$ preservers in both directions 
the family of lines contained in ${\mathcal P}_1$.
For distinct  $P_I,P_J\in {\mathcal P}_1$ we have $P_I\odot P_J\in {\mathcal P}_1$  
if and only if $|I\cap J|=1$.
This means that the restriction of $f$ to ${\mathcal P}_1$ induces an automorphism of 
the Johnson graph $J(n,2)$. 
Since $n\ge 7$,
this automorphism  is induced by a permutation on $[n]$ (see Appendix).
So, the restriction of $f$ to ${\mathcal P}_1$ is induced  by a coordinate permutation $s$ of $V$.

For every $Q\in {\mathcal P}_2$ we choose $P,P'\in {\mathcal P}_1$ such that 
$Q=P\odot P'$. Then $f(Q)$ is on the line $\langle f(P),f(P')\rangle$ and
$$f(Q)=f(P)\odot f(P')=s(P)\odot s(P')=s(P\odot P')=s(Q).$$
So, the restriction of $f$ to $\overline{{\mathcal P}}_1={\mathcal P}_1\cup {\mathcal P}_2$
is induced by $s$. 

In the case when $m\ge 3$, 
for every $Q\in {\mathcal P}_3$ we choose $P\in {\mathcal P}_1$ and $P'\in {\mathcal P}_2$ such that 
$Q=P\odot P'$. Then $f(Q)$ is on the line $\langle f(P),f(P')\rangle$ which implies that $f(Q)=s(Q)$.
Recursively, we establish that that $f(Q)=s(Q)$
for every $Q\in \cup^m_{i=1} {\mathcal P}_i$. 
\end{proof}

By Lemma \ref{lemma-lambda}, 
$f({\mathcal P}_1)={\mathcal P}_1$ if $n\ne 4m+\epsilon$ with $\epsilon\in \{-1,0,1\}$.
Hence, Theorem \ref{theorem-aut} is proved for the case when $n\ge 3m+1$ and $n$ is distinct from 
$4m+\epsilon$, $\epsilon\in \{-1,0,1\}$.

\subsection{The case $n=3m$}
In view of Subsection~\ref{subs:firststep}, $f$ can be extended to a bijective transformation of ${\mathcal P}_1\cup {\mathcal P}_m$
which preservers in both directions 
the family of lines contained in ${\mathcal P}_1\cup {\mathcal P}_m$ and containing
at least two points from ${\mathcal P}_m$. 
This extension is denoted by the same symbol $f$. 
Applying Lemma \ref{lemma-tex} to $P,P'\in {\mathcal P}_1$ such that $P\odot P'\in{\mathcal P}_1$
and using the reasoning from the proof of Lemma \ref{lemma-epsilon-case}, we
show that $f$ preservers in both directions the family of all lines contained in ${\mathcal P}_1$.
As in the proof of Lemma \ref{lemma-p1},
the restriction of $f$ to ${\mathcal P}_1$ is induced  by a coordinate permutation $s$ of $V$.
Then $f'=s^{-1}f$ is a bijective transformation of 
${\mathcal P}_1\cup {\mathcal P}_m$ 
whose restriction to ${\mathcal P}_1$ is identity
and $f'$ preserves
in both directions the family of lines contained in ${\mathcal P}_1\cup {\mathcal P}_m$ and containing
at least two points from ${\mathcal P}_m$. 

For every $2m$-element subset $I\subset [n]$ denote by ${\mathcal X}_I$ 
the set of  all $P_J\in {\mathcal P}_1$ such that $|I\cap J|=1$, or equivalently, 
$P_I\odot P_J\in{\mathcal P}_m$
(the line $\langle P_I,P_J\rangle$ contains two points from ${\mathcal P}_m$).
Observe that ${\mathcal P}_1\setminus {\mathcal X}_I$ is formed by all $P_J\in {\mathcal P}_1$ such that
$$J\subset I\quad\text{or}\quad J\subset [n]\setminus I.$$
Therefore, for $2m$-element subsets $I,I'\subset [n]$ 
we have ${\mathcal X}_I={\mathcal X}_{I'}$ if and only if $I=I'$. 
If $f'(P_I)=P_{I'}$ for some $2m$-element $I,I'\subset [n]$, 
then $${\mathcal X}_{I'}=f'({\mathcal X}_I)={\mathcal X}_I$$ which implies that $I=I'$.
So, $f'$ is identity and, consequently, $f$ is induced by $s$.

\subsection{The case $n=4m\pm 1$}\label{subsec:4m-pm-1}
Lemma \ref{lemma-epsilon-case} states that $f$ is induced by a linear automorphism $l$ of $S$.
Then Lemma \ref{lemma-lambda} guarantees that
\begin{equation}\label{eq-p}
l({\mathcal P}_1\cup{\mathcal P}_{(n-1)/2})={\mathcal P}_1\cup{\mathcal P}_{(n-1)/2}.
\end{equation}
By Lemma \ref{lemma-p1}, it suffices to consider the case when
there is $P_I\in{\mathcal P}_1$ such that  
$l(P_I)\in{\mathcal P}_{(n-1)/2}$,
in other words, $l(P_I)=P_{I'}$ and $|I'|=n-1$.
Without loss of generality we assume that $I=\{n-1,n\}$.

Let ${\mathcal X}$ be the subset of  all $P_J\in {\mathcal P}_1$ satisfying $J\subset [n-2]$. 
Note that $P_{\{n-1,n\}}\odot P_J$ belongs to ${\mathcal P}_2$ for every such $J$.
If $l(P_J)\in{\mathcal P}_{(n-1)/2}$, then 
$$l(P_{\{n-1,n\}}\odot P_J)=l(P_{\{n-1,n\}})\odot l(P_J)$$
is a point of ${\mathcal P}_1$ (since $P\odot P'\in{\mathcal P}_1$ for any distinct $P,P'\in {\mathcal P}_{(n-1)/2}$)
which contradicts \eqref{eq-p}.
Therefore, 
$l({\mathcal X})\subset {\mathcal P}_1$.

The restriction of $f$ to ${\mathcal X}$
can be considered as an embedding of $J(n-2,2)$ in $J(n,2)$ sending tops to tops. 
Since $n\ge 7$,
this embedding maps stars to subsets of stars which means that it is induced by a certain injection
of $[n-2]$ to $[n]$ (see Appendix). 
We assume that the restriction of $f$ to ${\mathcal X}$ is identity (otherwise, we consider $sf$, where $s$ is a suitable coordinate permutation).

Taking  $l(e_1)=e_1$ we extend $l$ to a linear automorphism of $V$ which we also denote by $l$.
Using the fact that $l$ leaves every $e_{\{i,j\}}$ with $i,j\in [n-2]$ fixed, we establish that $l(e_i)=e_i$ for every $i\in [n-2]$.

Recall that 
$l(P_{\{n-1,n\}})=P_{I'}$ and $|I'|=n-1$ by our assumption.
If $l(e_i)=e_{I_i}$ for $i\in \{n-1,n\}$, then $I'=I_{n-1}\mathbin{\triangle} I_n$ and one of the following possibilities is realized:
\begin{enumerate}[(a)]
\item
$I_{n-1},I_n$ are proper subsets of $[n]$ and $|I_{n-1}\cap I_n|=1$, $I_{n-1}\cup I_n=[n]$;

\item
$I_{n-1}\cap I_n=\emptyset$ and $I_{n-1}\cup I_n=I'$;

\item
one of $I_{n-1},I_n$, say $I_n$, coincides with $[n]$ and the other is a singleton.
\end{enumerate}
We show that (a) and (b) are impossible.

{\it Case} (a). Since $n\ge 7$, one of $I_{n-1},I_n$, say $I_n$, contains more than $3$ elements. 
If $i\in I_n$  is distinct from $n-1,n$, then
$$l(e_{\{i,n\}})=e_i+e_{I_n}=e_{I_n\setminus\{i\}}\quad\text{and}\quad 2<\bigl|I_n\setminus \{i\}\bigr|<n-1$$ 
which is impossible, since $l(P_{\{i,n\}})$ belongs to ${\mathcal P}_1$ or ${\mathcal P}_{(n-1)/2}$.

{\it Case} (b).
We use the above arguments if one of $I_{n-1},I_n$ contains more than $3$ elements.
If $|I_{n-1}|=|I_n|=3$ (this is possible if $n=7$) and $i\not \in I_n$, then
$$l(e_{\{i,n\}})=e_i+e_{I_n}=e_{I_n\cup\{i\}}\quad\text{and}\quad|I_n\cup \{i\}|=4<n-1,$$
a contradiction.

{\it Case} (c). We have $l(e_n)=e_{[n]}$.
Since $l(e_i)=e_i$ for every $i\le n-2$, we obtain that $l(e_{n-1})$ is $e_{n-1}$ or $e_n$. We assume that $l(e_{n-1})=e_{n-1}$ (in the second case, $f$ can be replaced by $sf$, where $s$ is the transposition of $(n-1)$-th and $n$-th coordinates).
Then $l=l_n$ if  $\epsilon=-1$ (Example \ref{exmp-lin1}). 
In the case when $\epsilon=1$,
the linear automorphism does not preserve ${\mathcal P}_m$.
Indeed, for every $2m$-element $I\subset [n]$ containing $n$ we have
$$l(e_I)=e_{[n]}+e_{I\setminus\{n\}}=e_{[n]\setminus (I\setminus \{n\})}$$
and
$$\bigl|[n]\setminus (I\setminus \{n\})\bigr|=4m+1-(2m-1)=2m+2.$$

\subsection{The case $n=4m$}\label{subsec:4m}
As in the previous subsection,
$f$ is induced by a linear automorphism $l$ of $S$
and
$$
l({\mathcal P}_1\cup{\mathcal P}_{2m-1})={\mathcal P}_1\cup{\mathcal P}_{2m-1}
$$
by Lemma \ref{lemma-lambda}.

\begin{lemma}\label{lemma-obs1}
If a line of ${\mathcal P}(S)$ is contained in ${\mathcal P}_1\cup{\mathcal P}_{2m-1}$,
then it is contained in ${\mathcal P}_1$ or it contains precisely two points from ${\mathcal P}_{2m-1}$.
\end{lemma}

\begin{proof}
Direct verification.
\end{proof}

Denote by $\Gamma$ the simple graph whose vertex set is ${\mathcal P}_1\cup {\mathcal P}_{2m-1}$ and 
distinct $P_I,P_J$ from this set are connected by en edge if $P_I\odot P_J$ belongs to ${\mathcal P}_1\cup {\mathcal P}_{2m-1}$.
This graph can be identified with the union of the Johnson graphs $J(n,2)$ and $J(n,n-2)$, where 
$2$-element and $(n-2)$-element subsets are connected by an edge if their intersection is a singleton. 
Every maximal clique of $J(n,t)$, $t=2,n-2$ consists of $n-1$ or $3$ elements (see Appendix).
A direct verification shows that there are the following four types of maximal cliques of $\Gamma$
containing precisely $n-1$ vertices:
\begin{enumerate}[(1)]
\item
the set of all $P_I$ such that $I$ is a $2$-element subset containing fixed $i\in [n]$ (corresponds to a star of $J(n,2)$);

\item
the set of all $P_I$ such that $I$ is an $(n-2)$-element subset contained in a 
fixed $(n-1)$-element subset of $[n]$ (corresponds to a top of $J(n,n-2)$);

\item
the set formed by a certain $P_I$, $|I|=n-2$ and all $P_J$, $|J|=2$ such 
that $|I\cap J|=1$ and $J$ contains fixed $i\in [n]\setminus I$;

\item
the set formed by a certain $P_I$, $|I|=2$ and all $P_J$, $|J|=n-2$ such 
that $J$ is contained  in a fixed $(n-1)$-element subset 
of $[n]$ which intersects $I$ in a Singleton.
\end{enumerate}
Note that $\Gamma$ contains other maximal cliques whose number of vertices differs from $n-1$.

\begin{lemma}\label{lemma-obs2}
If ${\mathcal A}$ is a maximal clique of type {\rm (1)}, then for every $Q\in {\mathcal P}_1\setminus {\mathcal A}$
there are precisely two $P,P'\in {\mathcal A}$ such that $Q,P,P'$ form a line.
\end{lemma}

\begin{proof}
Direct verification.
\end{proof}

Let ${\mathcal A}$ be the maximal clique of type (1) corresponding to a certain $i\in [n]$.
Then ${\mathcal P}_1\setminus {\mathcal A}$ is formed by 
all $P_I$, where $I$ is a $2$-element subset of $[n]\setminus \{i\}$.
Since the restriction of $f$ to ${\mathcal P}_1\cup{\mathcal P}_{2m-1}$
is an automorphism of $\Gamma$, there are the following four possibilities for $l({\mathcal A})$:
\begin{itemize}
\item
If  $l({\mathcal A})$ is a maximal clique of type (1),
then Lemmas \ref{lemma-obs1} and \ref{lemma-obs2} imply that 
$l({\mathcal P}_1\setminus {\mathcal A})$ is contained in ${\mathcal P}_1$ which means that 
$l({\mathcal P}_1)={\mathcal P}_1$ and, consequently, $f$ is induced by a coordinate permutation 
by Lemma \ref{lemma-p1}.

\item
If  $l({\mathcal A})$ is a maximal clique of type (2),
then $l({\mathcal P}_1\setminus {\mathcal A})$ is contained in ${\mathcal P}_1$ by 
Lemmas \ref{lemma-obs1} and \ref{lemma-obs2}. 

\item
If $l({\mathcal A})$ is a maximal clique of type (3),
then there is $j\in [n]\setminus\{i\}$ such that $l(P_{\{i,j\}})\in{\mathcal P}_{2m-1}$
and the remaining elements of $l({\mathcal A})$ belong to ${\mathcal P}_1$. 
Lemmas \ref{lemma-obs1} and \ref{lemma-obs2} show that $l(P_I)$ belongs to ${\mathcal P}_1$
for every $2$-element subset $I\subset [n]\setminus\{j\}$
and the remaining elements of $l({\mathcal P}_1)$ belong to ${\mathcal P}_{2m-1}$. 

\item
If $l({\mathcal A})$ is a maximal clique of type (4),
then there is $j\in [n]\setminus\{i\}$ such that $l(P_{\{i,j\}})\in{\mathcal P}_1$
and the remaining elements of $l({\mathcal A})$ belong to ${\mathcal P}_{2m-1}$.
Lemmas \ref{lemma-obs1} and \ref{lemma-obs2} imply that $l(P_I)$ belongs to ${\mathcal P}_1$
for every $2$-element subset $I\subset [n]\setminus \{i,j\}$
and the remaining elements of $l({\mathcal P}_1)$, except 
$l(P_{\{i,j\}})$, belong to ${\mathcal P}_{2m-1}$. 
\end{itemize}
So, we have to consider the following two cases:
\begin{enumerate}[(i)]
\item there is $i\in [n]$ such that 
$l(P_I)$ belongs to ${\mathcal P}_1$ for every $2$-element subset $I\subset [n]\setminus\{i\}$
and the remaining elements of $l({\mathcal P}_1)$ belong to ${\mathcal P}_{2m-1}$;
\item there are distinct $i,j\in [n]$ such that 
$l(P_I)$ belongs to ${\mathcal P}_1$ if $I=\{i,j\}$ or $I$ is a $2$-element subset of $[n]\setminus \{i,j\}$
and the remaining elements of $l({\mathcal P}_1)$ belong to ${\mathcal P}_{2m-1}$.
\end{enumerate}

{\it Case} (i). 
As in the previous subsection, we consider the restriction of $f$ to the set 
of all $P_I$ such that $I$ is a $2$-element subset of $[n]\setminus \{i\}$.
This restriction induces an embedding of $J(n-1,2)$ in $J(n,2)$.
Since $n\ge 8$, this embedding sends tops to tops and stars to subsets of stars 
which means that it is induced by an injection of $[n]\setminus \{i\}$ to $[n]$.
We can assume that this injection is identity, i.e. 
$l(P_I)=P_I$ for every $2$-element set $I\subset [n]\setminus \{i\}$. 
Taking  $l(e_j)=e_j$  for a certain $j\in [n]\setminus \{i\}$
we extend $l$ to a linear automorphism of $V$
and establish that $l(e_j)=e_j$ for all $j\in [n]\setminus \{i\}$.

Suppose that $l(e_i)=e_{I_i}$.
For every $j\in [n]\setminus \{i\}$ there is an $(n-2)$-element subset $I_j\subset [n]$ such that
$$e_{I_j}=l(e_{\{i,j\}})=e_j+e_{I_i}.$$
Then
$$I_j=I_i\setminus \{j\}\quad\text{if}\quad j\in I_i\qquad\text{and}\qquad 
I_j=I_i\cup \{j\}\quad\text{if}\quad j\not\in I_i.$$
Since $|I_i|>1$ (otherwise, $f({\mathcal P}_1)={\mathcal P}_1$), we take any $j\in \big([n]\setminus \{i\}\big)\cap I_i$
and establish that $|I_i|=n-1$. 
Furthermore,  $I_i=[n]\setminus \{i\}$
(indeed, if there is $j\in \big([n]\setminus \{i\}\big)\setminus I_i$, then
$I_j=I_i\cup \{j\}$ coincides with $[n]$ which is impossible).

So, $l(e_i)=e_{[n]\setminus \{i\}}$.
Then for every subset $I\subset [n]$ we have
$l(e_I)=e_I$ if $i\not\in I$ and 
$$l(e_I)=e_{[n]\setminus\{i\}}+e_{I\setminus\{i\}}=e_{[n]\setminus I}$$
if $i\in I$.
This implies that $l(e_I)=l(e_{[n]\setminus I})$, i.e.\ $l$ is not injective and we get a contradiction.

{\it Case} (ii). 
As in the previous case, we assume that $l(P_I)=P_I$ for every $2$-element subset $I\subset [n]\setminus\{i,j\}$.
Since $P_{\{i,j\}}$ 
is the unique point of the geometry ${\mathcal P}_1$
non-collinear to all $P_I\in {\mathcal P}_1$ with
$I\subset [n]\setminus\{i,j\}$,
we obtain that
\begin{equation}\label{eq-ij}
l(P_{\{i,j\}})=P_{\{i,j\}}.
\end{equation}
We extend $l$ to a linear automorphism of $V$ taking $l(e_t)=e_t$ for a certain $t\in [n]\setminus\{i,j\}$
and obtain that $l(e_t)=e_t$ for all $t\in [n]\setminus\{i,j\}$.
Then \eqref{eq-ij} shows that one of $l(e_i),l(e_j)$  is $e_{J\cup\{i\}}$ and the other is $e_{J\cup\{j\}}$
for a certain subset $J\subset [n]$ which contains neither of $i$ and $j$.
Using the fact that $l(P_{\{i,t\}})$ and $l(P_{\{j,t\}})$ belong to ${\mathcal P}_{2m-1}$ for all $t\in [n]\setminus\{i,j\}$
we establish that $$J=[n]\setminus\{i,j\}.$$
Therefore, one of $l(e_i),l(e_j)$  is $[n]\setminus\{i\}$ and the other is $[n]\setminus\{j\}$
which means that $l$ is $s_{ij}$ or $s'_{ij}$ (Example~\ref{exmp-lin3}).

\section{Remarks on non-binary case}
%The non-binary case is more complicated.
In this section, we assume that ${\mathbb F}$ is the $q$-element field with $q\ge 3$
and $V$ is the $n$-dimensional vector space ${\mathbb F}^n$ over ${\mathbb F}$.
For every integer $t$ satisfying $1\le t\le n$ we denote by $V_t$
the set of all vectors of $V$ whose Hamming weight is $t$. 
First of all, we observe that every $V_t$ spans $V$.
Indeed, for every $i\in [n]$ we can choose vectors 
$x=(x_1,\dots,x_n),y=(y_1,\dots,y_n)\in V_t$ such that $x_j=y_j$ if $j\ne i$ and $x_i-y_i=1$
which implies that $x-y=e_i$ (as above, $e_i$ is the vector from the standard basis of $V$ whose $i$-coordinate is non-zero).

Consider the geometry ${\mathcal H}_t$ whose points are $1$-dimensional subspaces 
spanned by vectors from $V_t$ and whose lines are lines of the projective space ${\mathcal P}(V)$
contained in ${\mathcal H}_t$.
We do not deal with all cases when such lines exist. Suppose that
%\begin{equation}\label{eq-simp}
$$
n=\frac{q^k-1}{q-1}\quad\text{and}\quad t=q^{k-1}.
$$
%\end{equation}
Then the maximal singular subspaces of ${\mathcal H}_t$
correspond to $q$-ary simplex codes of dimension $k$.
The geometry of $4$-ary simplex codes of dimension $2$ is investigated in \cite{KP-dim2}
and some remarks on the general case can be found in \cite[Section 6]{KP-JG}.

Two distinct points $\langle x\rangle, \langle y\rangle$ 
are collinear in ${\mathcal H}_t$  if and only if 
the $(2\times n)$-matrix $\binom{x}{y}$ 
can be extended to a generator matrix of a simplex code
which  is equivalent to the fact that in the matrix $\binom{x}{y}$ every non-zero column 
is proportional to precisely  $q^{k-2}$ columns 
including itself and the matrix contains precisely $\frac{q^{k-2}-1}{q-1}$ zero columns;
see \cite[Section 3]{KP-JG} for the details. 
Using this observation we can construct linearly independent $x,y\in V_t$ such that 
the Hamming distance between $x$ and $y$ is $t$ (i.e. $x-y\in V_t$),
but $\langle x \rangle, \langle y \rangle$ are non-collinear points of ${\mathcal H}_t$.
Also, we can find $x',y' \in V_t$ 
such that the line of ${\mathcal P}(V)$ connecting  $\langle x' \rangle$ and $\langle y' \rangle$
does not contain other points of ${\mathcal H}_t$.
Since lines of ${\mathcal P}(V)$ connecting non-collinear points of ${\mathcal H}_t$
can contain more than one point from ${\mathcal P}(V)\setminus {\mathcal H}_t$, we cannot use arguments
from Subsection 4.1.

The collinearity graph of ${\mathcal H}_t$ is connected.
The following example shows that the diameter of this graph 
can be greater than $2$.

\begin{exmp}
Suppose that $q=5$ and $k=2$.
Then
$$P=\langle (0,1,1,1,1,1)\rangle\quad\text{and}\quad P'=\langle (0,1,1,1,2,2)\rangle$$
are non-collinear points of ${\mathcal H}_t$. 
Show that there is no point of ${\mathcal H}_t$ collinear to both $P,P'$.
If such a point $Q=\langle (1,q_0,\ldots,q_4)\rangle$ exists
(it is clear that the first coordinates of vectors from $Q$ are non-zero), then
in each of the matrices 
$$
M=\begin{bmatrix}
0 & 1 & 1 & 1 &1 &1 \\
1 & q_0 & q_1& q_2 & q_3 & q_4
\end{bmatrix}, \qquad
M'=\begin{bmatrix}
0 & 1 & 1 & 1 &2 &2 \\
1 & q_0 & q_1& q_2 & q_3 & q_4
\end{bmatrix}$$
the columns are mutually non-proportional.
Hence $q_0,\ldots,q_4$ are mutually distinct elements of the field ${\mathbb F}$.
If one of $q_3,q_4$, say $q_3$, is zero, then the last column of $M'$ is proportional to
the second, third or fourth column of $M'$
(indeed, $2^{-1}q_4=3q_4$ is non-zero and distinct from $q_4$,
consequently, it coincides with $q_0,q_1$ or $q_2$), a contradiction.
In the case when $q_3, q_4\neq 0$, none of the last two columns 
is proportional to one of the remaining columns of $M'$ only if $3q_3=q_4$ and $3q_4=q_3$, but this implies that $3q_4=q_3=2q_4$ which is impossible.
\end{exmp}

\section{Appendix}
The {\it Johnson graph} $J(n,t)$ is the simple graph whose vertex set is formed by all $t$-element subsets of $[n]$ and two such subsets are connected by an edge if their intersection consists of $t-1$ elements.
This graph is connected and 
the complementary map $I\to [n]\setminus I$ induces an isomorphism between $J(n,t)$ and $J(n,n-t)$. 
If $t=1,n-1$, then $J(n,t)$ is a complete graph. In the case when $1<t<n-1$, each maximal clique of $J(n,t)$ is of one of the following two types:
\begin{itemize}
\item the {\it star} ${\mathcal S}(I)$ formed by all $t$-element subsets containing a fixed $(t-1)$-element subset $I$,
\item the {\it top} ${\mathcal T}(J)$ formed by all $t$-element subsets contained in a fixed $(t+1)$-element subset $J$.
\end{itemize}
Stars and tops contain $n-t+1$ and $t+1$ elements, respectively.
The complementary map sends stars to tops and tops to stars.

If $n\ne 2t$, then every automorphism of $J(n,t)$ is induced by a permutation on $[n]$.
In the case when $n=2t$, every automorphism of $J(n,t)$ is induced by a permutation on $[n]$
or it is the composition of the automorphism induced by a permutation and the complementary map.
In particular, if  $1<t<n-1$, then every automorphism of $J(n,t)$ preserving the types of maximal cliques is induced
by a permutation on $[n]$.

In the cases $n=4m\pm 1$ and $n=4m$ considered in subsections \ref{subsec:4m-pm-1} and \ref{subsec:4m} we make use of the simple fact involving embeddings of Johnson graphs.
\begin{fact}
For $m<n$, every embedding of $J(m,2)$ in $J(n,2)$ 
sending tops to tops and stars to subsets of stars is induced by an injection of $[m]$ to $[n]$.
\end{fact}

To prove Proposition \ref{prop-conn} we need the following generalized version of 
the Johnson graph. 
Let $n,t,i$ be natural numbers such that 
$$1<t<n\quad\text{and}\quad\max\{0, 2t-n\}\le i<t.$$
Also, we require that $n>2t$ if $i=0$.
Consider the graph $J(n,t,i)$ whose vertices are $t$-element subsets of 
$[n]$ and two such subsets are connected by an edge if their intersection 
contains precisely $i$ elements.
This is the Johnson graph if $i=t-1$ and the Kneser graph if $n>2t$ and $i=0$
(in the case when $n=2t$ and $i=0$, we obtain the disjoint union of $K_2$).

\begin{lemma}\label{lemma-John}
The graph $J(n,t,i)$ is connected. 
\end{lemma}

\begin{proof}
Since the Johnson and Kneser graphs are connected, we assume that 
$0<i<t-1$.
Let $X$ and $Y$ be distinct $t$-element subsets of $[n]$.
It is sufficient to show that $X$ and $Y$ are connected by a path in $J(n,t,i)$
when $|X\cap Y|=t-1$, i.e. $X,Y$ are adjacent vertices in the Johnson graph. 
Suppose that this happens and 
$$X\setminus Y=\{a\},\qquad Y\setminus X=\{b\}.$$
We take 
an $(i-1)$-element subset $A\subset X\cap Y$ and a $(t-i-1)$-element subset 
$$B\subset [n]\setminus (X\cup Y);$$
the second subset exists, since $\max\{0, 2t-n\}\le i$ and 
$$\bigl|[n]\setminus (X\cup Y)\bigr|=n-t-1\ge t-i-1.$$
Then 
$$Z=\{a,b\}\cup A\cup B$$
intersects each of $X,Y$ precisely in an $i$-element subset.
\end{proof}


\begin{thebibliography}{99}
\bibitem{Bonis} 
A. Bonisoli, {\it Every equidistant linear code is a sequence of dual Hamming codes},
Ars Combin. 18(1984), 181--186.

\bibitem{Cameron}
P.J. Cameron, {\it Embedding partial Steiner triple systems so that their automorphisms extend},
J. Comb. Des. 13(6), 466--470, 2005.

\bibitem{CGK}
I. Cardinali, L. Giuzzi, M. Kwiatkowski,
{\it On the Grassmann graph of linear codes}, 
Finite Fields Appl. 75 (2021) 101895.

\bibitem{CG}
I. Cardinali, L. Giuzzi,
{\it Grassmannians of codes}, Finite Fields Appl., 94(2024), art. 102342.


\bibitem{HP-book} 
W.C. Huffman, V. Pless, {\it Fundamentals of Error-Correcting Codes}, Cambridge University Press, 2003.


\bibitem{KP1} 
M. Kwiatkowski, M. Pankov, {\it On the distance between linear codes}, Finite Fields Appl. 39(2016), 251--263.

\bibitem{KP2} 
M. Kwiatkowski, M. Pankov, {\it Chow's theorem for linear codes}, Finite Fields Appl. 46(2017) 147--162.

\bibitem{KPP}
M. Kwiatkowski, M. Pankov, A. Pasini,
{\it The graphs of projective codes},  Finite Fields Appl. 54(2018), 15--29.

\bibitem{KP-dim2}
M. Kwiatkowski, M. Pankov,
{\it The graph of 4-ary simplex codes of dimension 2}, Finite Fields Appl. 67(2020), art. 101709.

\bibitem{KP-JG}
M. Kwiatkowski, M. Pankov,
{\it On maximal cliques in the graph of simplex codes}, J. Geom. 115(2024), art. 10.

\bibitem{Pankov-book} M. Pankov,
{\it Geometry of semilinear embeddings:Relations to graphs and codes}, World Scientific 2015. 

\bibitem{Pank}
M. Pankov, {\it The graphs of non-degenerate linear codes}, J. Comb. Theory, Ser.  A 195(2023), art. 105720.


\bibitem{embPSTS} 
C.A. Rodger, S.J. Stubbs {\it Embedding partial triple systems}, J. Comb. Theory Ser. A. 44, 241--252, 2014.

\bibitem{Ryser}
H.J. Ryser,  {\it An extension of a theorem of de Bruijn and Erd\H{o}s on combinatorial
designs}, Journal of Algebra 10 (1968), 246-261.

\bibitem{TVN} 
M. Tsfasman, S. Vl\v{a}du\c{t}, D. Nogin, {\it Algebraic Geometry Codes. Basic notions}, Amer. Math. Soc., Providence, 2007.

\bibitem{Ward} H.N. Ward {\it An Introduction to Divisible Codes}, 
Designs, Codes and Cryptography, 17(1999), 73-79.



\end{thebibliography}
\end{document}